\documentclass[a4paper,12pt,reqno]{amsart}
\usepackage{eurosym}
\usepackage{amsfonts}
\usepackage{amsmath,amsthm,amssymb}

\nonstopmode \numberwithin{equation}{section}

\newtheorem{theorem}{Theorem}

\newtheorem{lemma}{Lemma}[section]

\allowdisplaybreaks

\begin{document}
\title[Fractional operator representations of
generalized Struve function]{Marichev-Saigo-Maeda fractional operator representations of
generalized Struve function}

\begin{abstract}
The aim of this paper is to apply generalized operators of fractional integration and differentiation involving Appell's function $F_3(:)$ due to Marichev-Saigo-Maeda, to
the generalized Struve function. The results are expressed in terms of generalized
Wright function. The results obtained here are general in nature and can easily obtain various known results.
\end{abstract}

\author{K.S. Nisar, Abdon Atangana and S.D. Purohit}
\address{K. S. Nisar : Department of Mathematics, College of Arts and
Science, Prince Sattam bin Abdulaziz University, Wadi Aldawaser, Riyadh
region 11991, Saudi Arabia}
\email{ksnisar1@gmail.com, n.sooppy@psau.edu.sa}
\address{Abdon Atangana : Institute for Ground water Studies, Faculty of Natural and Agricultural Sciences, University of the Free State, Bloemfontein 9300, South Africa}
\email{abdonatangana@yahoo.fr}
\address{S.D. Purohit : Department of HEAS (Mathematics), 
Rajasthan Technical University, Kota 324010, Rajasthan, India.}
\email{sunil\_a\_purohit@yahoo.com}

\keywords{Marichev-Saigo-Maeda fractional integral and differential
operators, Generalized Struve function, Generalized Wright function\\
\textbf{AMS 2010 Subject Classification:} 26A33, 33C05, 33C10, 33C20.}
\maketitle

\section{Introduction}

Recently, Nisar {\it{et al.}} \cite{Nisar} introduced a new generalization of Struve
function and defined as the following power series: 
\begin{equation}
_{a}\mathcal{W}_{p,b,c,\xi }^{\alpha ,\mu }\left( z\right)
:=\sum_{k=0}^{\infty }\frac{\left( -c\right) ^{k}}{\Gamma \left( \alpha
k+\mu \right) \Gamma \left( ak+\frac{p}{\xi }+\frac{b+2}{2}\right) }\left( 
\tfrac{z}{2}\right) ^{2k+p+1} ~~~~(a\in \mathbb{N}, p, b, c\in \mathbb{C}),  \label{1}
\end{equation}%
where $\lambda >0,\alpha >0,\xi >0$ and $\mu $ is an arbitrary parameter.

The Fox-Wright hypergeometric function ${}_{p}\Psi _{q}(z)$ is given by the
series 
\begin{equation}
{}_{p}\Psi _{q}(z)={}_{p}\Psi _{q}\left[ 
\begin{array}{c}
(a_{i},\alpha _{i})_{1,p} \\ 
(b_{j},\beta _{j})_{1,q}%
\end{array}%
\bigg|z\right] =\frac{\prod_{j=1}^{q}\Gamma (\beta _{j})}{%
\prod_{i=1}^{p}\Gamma (\alpha _{i})}\displaystyle\sum_{k=0}^{\infty }\dfrac{%
\prod_{i=1}^{p}\Gamma (a_{i}+\alpha _{i}k)}{\prod_{j=1}^{q}\Gamma
(b_{j}+\beta _{j}k)}\dfrac{z^{k}}{k!},  \label{Fox-Wright}
\end{equation}%
where $a_{i},b_{j}\in \mathbb{C}$, and $\alpha _{i},\beta _{j}\in \mathbb{R}$
($i=1,2,\ldots ,p;j=1,2,\ldots ,q$). Asymptotic behavior of this function
for large values of argument of $z\in {\mathbb{C}}$ were studied in \cite%
{Foxc}, under the condition 
\begin{equation}
\displaystyle\sum_{j=1}^{q}\beta _{j}-\displaystyle\sum_{i=1}^{p}\alpha
_{i}>-1,  \label{FW-C}
\end{equation}%
was found in the work of \cite{Wright-2,Wright-3}. Properties of this
generalized Wright function were investigated in \cite{Kilbas} (see also 
\cite{Kilbas-Book, Kilbas-itsf, Kilbas-frac}).

The familiar generalized hypergeometric function $_{p}F_{q}$ is defined as follows 
\cite{Rainville}: 
\begin{equation}  \label{eqn-1-hyper}
_{p}F_{q}\left[ 
\begin{array}{r}
\left( a _{p}\right) ; \\ 
\left( b _{q}\right) ;%
\end{array}
z\right] =\sum\limits_{n=0}^{\infty }\frac{\Pi _{j=1}^{p}\left( a_{j}\right)
_{n}}{\Pi _{j=1}^{q}\left( b _{j}\right) _{n}}\frac{z^{n}}{n! }
\end{equation}
\begin{equation*}
(p\leq q,\, z \in \mathbb{C};\,\, p=q+1,\, \left\vert z\right\vert <1),
\end{equation*}
which is an obvious special case of the Fox-Wright hypergeometric function $%
{}_{p}\Psi _{q}(z)$ \eqref{Fox-Wright} when $\alpha _{i}=1=\beta _{j}$ ($%
i=1,2,\ldots ,p;j=1,2,\ldots ,q$).

\vskip3mm Let $\lambda $, $\lambda ^{\prime }$, $\xi $, $\xi ^{\prime }$, $%
\gamma \in \mathbb{C}$ with $\Re (\gamma )>0$ and $x\in \mathbb{R}^{+}$.
Then the generalized fractional integral operators involving the Appell
functions $F_{3}$ are defined as follows: 
\begin{equation}
\left( I_{0{+}}^{\lambda ,\lambda ^{\prime },\xi ,\xi ^{\prime },\gamma
}f\right) (x)=\frac{x^{-\lambda }}{\Gamma (\gamma )}\int_{0}^{x}(x-t)^{%
\gamma -1}t^{-\lambda ^{\prime }}F_{3}\left( \lambda ,\lambda ^{\prime },\xi
,\xi ^{\prime };\gamma ;1-\frac{t}{x},1-\frac{x}{t}\right) f(t)\,\mathrm{d}t
\label{Int-1}
\end{equation}%
and 
\begin{equation}
\left( I_{-}^{\lambda ,\lambda ^{\prime },\xi ,\xi ^{\prime },\gamma
}f\right) (x)=\frac{x^{-\lambda ^{\prime }}}{\Gamma (\gamma )}%
\int_{x}^{\infty }(t-x)^{\gamma -1}t^{-\lambda }F_{3}\left( \lambda ,\lambda
^{\prime },\xi ,\xi ^{\prime };\gamma ;1-\frac{t}{x},1-\frac{x}{t}\right)
f(t)\,\mathrm{d}t.  \label{Int-2}
\end{equation}%
The generalized fractional integral operators of the types \eqref{Int-1} and %
\eqref{Int-2} have been introduced by Marichev \cite{Marichev} and later
extended and studied by Sagio and Maeda \cite{Saigo}. Recently, Purohit {\it{et al.}} \cite{Purohit}, Kumar {\it{et al.}} \cite{Kumar}, Baleanu {\it{et al.}} \cite{Baleanu} and Mondal and Nisar \cite{SRN} have investigated image formulas for Marichev-Saigo-Maeda fractional integral operators involving various special functions.

The corresponding fractional differential operators have their respective
forms: 
\begin{equation}
\left( D_{0+}^{\lambda ,\lambda ^{\prime },\xi ,\xi ^{\prime },\gamma
}f\right) \left( x\right) =\left( \frac{\mathrm{d}}{\mathrm{d}x}\right) ^{%
\left[ \Re \left( \gamma \right) \right] +1}\left( I_{0+}^{-\lambda ^{\prime
},-\lambda ,-\xi ^{\prime }+\left[ \Re \left( \gamma \right) \right] +1,-\xi
,-\gamma +\left[ \Re \left( \gamma \right) \right] +1}f\right) \left(
x\right)  \label{Dif-1}
\end{equation}%
and 
\begin{equation}
\left( D_{-}^{\lambda ,\lambda ^{\prime },\xi ,\xi ^{\prime },\gamma
}f\right) \left( x\right) =\left( -\frac{\mathrm{d}}{\mathrm{d}x}\right) ^{%
\left[ \Re \left( \gamma \right) \right] +1}\left( I_{-}^{-\lambda ^{\prime
},-\lambda ,-\xi ^{\prime },-\xi +\left[ \Re \left( \gamma \right) \right]
+1,-\gamma +\left[ \Re \left( \gamma \right) \right] +1}f\right) \left(
x\right) .  \label{Dif-2}
\end{equation}%
The fractional integral operators have many interesting applications in
various fields, for example certain class of complex analytic
functions (see \cite{Kim}). For some basic results on fractional calculus,
one may refer to \cite{Kiryakova, Miller, Srivastava2}.

The following known results will be required (see \cite{Saigo}, \cite%
{Kilbas-itsf}). \vskip3mm

\begin{lemma}
\label{lem-1} Let $\lambda ,\lambda ^{\prime },\xi ,\xi ^{\prime },\gamma
,\rho \in \mathbb{C}$ be such that $Re{(\gamma)}>0$ and 
\begin{equation*}
\Re{(\rho)}>\max \{0,\Re{(\lambda-\lambda'-\zeta-\gamma)},%
\Re{(\lambda'-\zeta')}\}.
\end{equation*}%
then there exists the relation 
\begin{equation}
\left( I_{0+}^{\lambda ,\lambda ^{\prime },\xi ,\xi ^{\prime },\gamma
}\;t^{\rho -1}\right) (x)=\frac{\Gamma \left( \rho \right) \Gamma \left(
\rho +\gamma -\lambda -\lambda ^{\prime }-\xi \right) \Gamma \left( \rho
+\xi {^{\prime }}-\lambda {^{\prime }}\right) }{\Gamma \left( \rho +\xi {%
^{\prime }}\right) \Gamma \left( \rho +\gamma -\lambda -\lambda {^{\prime }}%
\right) \Gamma \left( \rho +\gamma -\lambda {^{\prime }}-\xi \right) }%
x^{\rho -\lambda -\lambda ^{\prime }+\gamma -1}
\end{equation}%
where 
\begin{equation*}
\Gamma \left[ 
\begin{array}{l}
a,b,c \\ 
d,e,f%
\end{array}%
\right] =\frac{\Gamma {(a)}\Gamma {(b)}\Gamma {(c)}}{\Gamma {(d)}\Gamma {(e)}%
\Gamma {(f)}}.
\end{equation*}
\end{lemma}

\begin{lemma}
\label{lem-2} Let $\lambda $, $\lambda ^{\prime }$, $\xi $, $\xi ^{\prime }$%
, $\gamma $, $\rho \in \mathbb{C}$ such that $\Re (\gamma )>0$ and 
\begin{equation*}
\Re {(\rho )}>\max \{\Re {(\xi )},\,\Re {(-\lambda -\lambda ^{\prime
}+\gamma )},\,\Re {(-\lambda -\xi ^{\prime }+\gamma )}\}.
\end{equation*}%
Then the following formula holds true: 
\begin{eqnarray}
&&\left( I_{-}^{\lambda ,\lambda ^{\prime },\xi ,\xi ^{\prime },\gamma
}t^{-\rho }\right) (x)  \label{Lem2-Eq2} \\
&=&\frac{\Gamma \left( -\xi +\rho \right) \Gamma \left( \lambda +\lambda
^{\prime }-\gamma +\rho \right) \Gamma \left( \lambda +\xi {\ ^{\prime }}%
-\gamma +\rho \right) }{\Gamma \left( \rho \right) \Gamma \left( \lambda
-\xi +\rho \right) \Gamma \left( \lambda +\lambda {^{\prime }+\xi }%
^{^{\prime }}-\gamma +\rho \right) }x^{-\lambda -\lambda ^{\prime }+\gamma
-\rho }.  \notag
\end{eqnarray}
\end{lemma}

The aim of this paper is to apply the generalized operators of fractional calculus for the generalized Struve function in order to get certain new image formulas.

\section{Fractional integrals of generalized Struve function}

\begin{theorem}
\label{Th1}Let $\lambda ,\lambda ^{\prime },\xi ,\xi ^{\prime },\gamma ,\rho
\in \mathbb{C}$ be such that $Re{(\gamma )}>0$ and 
\begin{equation*}
\Re{(\rho)}>\max \{0,\Re{(\lambda-\lambda'-\zeta-\gamma)},%
\Re{(\lambda'-\zeta')}\}.
\end{equation*}%
then 
\begin{eqnarray*}
&&\left( I_{0+}^{\lambda ,\lambda ^{\prime },\xi ,\xi ^{\prime },\gamma
}t^{\rho -1}~_{a}\mathcal{W}_{p,b,c,\xi }^{\alpha ,\mu }\left( t\right)
\right) (x) \\
&=&\frac{x^{\rho +p-\lambda -\lambda ^{^{\prime }}+\gamma }}{2^{p+1}} \\
&&\times _{4}\Psi _{5}\left[ 
\begin{array}{c}
\left( \rho +p+1,2\right) ,\left( \rho +p+\gamma -\lambda -\lambda
^{^{\prime }}-\xi ,2\right) , \\ 
\left( \rho +p+1+\xi ^{^{\prime }},2\right) ,\left( \rho +p+1+\gamma
-\lambda -\lambda ^{^{\prime }},2\right) ,%
\end{array}%
\right.  \\
&&~~~~~~~~~\ ~~~~~~~~\left. 
\begin{array}{c}
\left( \rho +p+1+\xi ^{^{\prime }}-\lambda ^{^{\prime }},2\right) ,\left(
1,1\right) ; \\ 
\left( \rho +p+1+\gamma +\lambda ^{^{\prime }}-\xi ,2\right) ,\left( \mu
,\alpha \right) ,\left( \frac{p}{\xi }+\frac{b+2}{2},a\right) ;%
\end{array}%
-\frac{cx^{2}}{4}\right].  
\end{eqnarray*}
\end{theorem}

\begin{proof}
\label{Pf1} Applying the definition of generalized of Struve function, we get
\begin{eqnarray*}
&&\left( I_{0+}^{\lambda ,\lambda ^{\prime },\xi ,\xi ^{\prime },\gamma
}t^{\rho -1}~_{a}\mathcal{W}_{p,b,c,\xi }^{\alpha ,\mu }\left( t\right)
\right) (x) \\
&=&\left( I_{0+}^{\lambda ,\lambda ^{\prime },\xi ,\xi ^{\prime },\gamma
}t^{\rho -1}\sum_{k=0}^{\infty }\frac{\left( -c\right) ^{k}}{\Gamma \left(
\alpha k+\mu \right) \Gamma \left( ak+\frac{p}{\xi }+\frac{b+2}{2}\right) }%
\left( \tfrac{t}{2}\right) ^{2k+p+1}\right)\left( x\right),
\end{eqnarray*}%
then on interchanging the integration and summation, we obtain
\begin{eqnarray*}
&&\left( I_{0+}^{\lambda ,\lambda ^{\prime },\xi ,\xi ^{\prime },\gamma
}t^{\rho -1}~_{a}\mathcal{W}_{p,b,c,\xi }^{\alpha ,\mu }\left( t\right)
\right) (x) \\
&=&\left( \sum_{k=0}^{\infty }\frac{\left( -c\right) ^{k}\left( 2\right)
^{-\left( 2k+p+1\right) }}{\Gamma \left( \alpha k+\mu \right) \Gamma \left(
ak+\frac{p}{\xi }+\frac{b+2}{2}\right) }I_{0+}^{\lambda ,\lambda ^{\prime
},\xi ,\xi ^{\prime },\gamma }t^{\rho +2k+p}\right) \left( x\right).
\end{eqnarray*}%
Now, for any $k=0,1,2,...$ and 
\begin{equation*}
\Re\left( l+\rho +2k+1\right) \geq \Re\left( \rho +l+1\right) >\max \left[ 0,%
\Re\left( \lambda -\lambda ^{^{\prime }}-\xi -\gamma \right) ,\Re\left(
\lambda ^{^{\prime }}-\xi ^{^{\prime }}\right) \right],
\end{equation*}%
we get%
\begin{eqnarray*}
&&\left( I_{0+}^{\lambda ,\lambda ^{\prime },\xi ,\xi ^{\prime },\gamma
}t^{\rho -1}~_{a}\mathcal{W}_{p,b,c,\xi }^{\alpha ,\mu }\left( t\right)
\right) (x) \\
&=&\sum_{k=0}^{\infty }\frac{\left( -c\right) ^{k}\left( 2\right) ^{-\left(
2k+p+1\right) }}{\Gamma \left( \alpha k+\mu \right) \Gamma \left( ak+\frac{p%
}{\xi }+\frac{b+2}{2}\right) } \\
&&\times \frac{\Gamma \left( \rho +p+1+2k\right) \Gamma \left( \rho
+p+1+\gamma -\lambda -\lambda ^{^{\prime }}-\xi +2k\right) }{\Gamma \left(
\rho +p+1+\xi ^{^{\prime }}+2k\right) \Gamma \left( \rho +p+1+\gamma
-\lambda -\lambda ^{^{\prime }}+2k\right) } \\
&&\times \frac{\Gamma \left( \rho +p+1+\xi ^{^{\prime }}-\lambda ^{^{\prime
}}+2k\right) }{\Gamma \left( \rho +p+1+\gamma -\lambda ^{^{\prime }}-\xi
+2k\right) }x^{\rho +p-\lambda -\lambda ^{^{\prime }}+\gamma +2k}.
\end{eqnarray*}
Interpreting the right-hand side of the above equation, in view of the definition
\eqref{Fox-Wright}, we arrive at the result of Theorem 1.
\end{proof}

\begin{theorem}
\label{Th2} Suppose $\lambda ,\lambda ^{\prime },\xi ,\xi ^{\prime },\gamma ,\rho
,b,c\in \mathbb{C}$, $a\in \mathbb{N}$ be such that $\frac{l}{\xi }+%
\frac{b}{2}\neq -1,-2,-3,..$ and 
\begin{equation*}
\Re{(\rho)}>\max \{\Re{(-\lambda-\lambda'+\gamma)},\Re{(-\lambda'-\zeta'+%
\gamma)},\Re{(\zeta)}\},
\end{equation*}%
then 
\begin{eqnarray*}
&&\left( I_{-}^{\lambda ,\lambda ^{\prime },\xi ,\xi ^{\prime },\gamma
}t^{\rho -1}~_{a}\mathcal{W}_{p,b,c,\xi }^{\alpha ,\mu }\left( t\right)
\right) (x) \\
&=&\frac{x^{-\lambda -\lambda ^{^{\prime }}+\gamma -\rho -p-1}}{2^{p+1}} \\
&&\times _{4}\Psi _{5}\left[ 
\begin{array}{c}
\left( -\xi +\rho +p+1,2\right) ,\left( \lambda +\lambda ^{^{\prime
}}-\gamma +\rho +p+1,2\right) , \\ 
\left( \rho +p+1,2\right) ,\left( \lambda -\xi -\rho +p+1,2\right) ,%
\end{array}%
\right.  \\
&&~~~~~~~~~\ ~~~~~~~~\left. 
\begin{array}{c}
\left( \lambda +\xi ^{^{\prime }}-\gamma +\rho +p+1^{^{\prime }},2\right)
,\left( 1,1\right) ; \\ 
\left( \lambda +\lambda ^{^{\prime }}+\xi ^{^{\prime }}-\gamma +\rho
+p+1,2\right) ,\left( \mu ,\alpha \right) ,\left( \frac{p}{\xi }+\frac{b+2}{2%
},a\right) ;%
\end{array}%
-\frac{cx^{2}}{4}\right]. 
\end{eqnarray*}
\end{theorem}

\begin{proof}
\label{Pf2} By making use of \eqref{1} in the integrand of Theorem 2 and interchanging the order of integral sign and summation, which is verified by uniform
convergence of the series, we find 
\begin{equation*}
\left( I_{-}^{\lambda ,\lambda ^{\prime },\xi ,\xi ^{\prime },\gamma
}t^{\rho -1}~_{a}\mathcal{W}_{p,b,c,\xi }^{\alpha ,\mu }\left( t\right)
\right) (x) \\
=\left( \sum_{k=0}^{\infty }\frac{\left( -c\right) ^{k}\left( 2\right)
^{-\left( 2k+p+1\right) }}{\Gamma \left( \alpha k+\mu \right) \Gamma \left(
ak+\frac{p}{\xi }+\frac{b+2}{2}\right) }I_{-}^{\lambda ,\lambda ^{\prime
},\xi ,\xi ^{\prime },\gamma }t^{\rho +2k+p}\right) \left( x\right).
\end{equation*}%
Using Lemma 1.2, we get%
\begin{eqnarray*}
&=&\sum_{k=0}^{\infty }\frac{\left( -c\right) ^{k}\left( 2\right) ^{-\left(
2k+p+1\right) }}{\Gamma \left( \alpha k+\mu \right) \Gamma \left( ak+\frac{p%
}{\xi }+\frac{b+2}{2}\right) } \\
&&\times \frac{\Gamma \left( -\xi +\rho +p+1+2k\right) \Gamma \left( \lambda
+\lambda ^{^{\prime }}-\gamma +\rho +p+1+2k\right) }{\Gamma \left( \rho
+p+1+2k\right) \Gamma \left( \lambda -\xi -\rho +p+1+2k\right) } \\
&&\times \frac{\Gamma \left( \lambda +\xi ^{^{\prime }}-\gamma +\rho
+p+1^{^{\prime }}+2\right) }{\Gamma \left( \lambda +\lambda ^{^{\prime
}}+\xi ^{^{\prime }}-\gamma +\rho +p+1+2k\right)}x^{-\lambda -\lambda ^{^{\prime }}+\gamma -\rho -p-1+2k},
\end{eqnarray*}%
which in view of Fox-Wright function arrive at the desired result.
\end{proof}

\section{Fractional differentials of generalized Struve function}

\label{Sec3}

Here we derive the Marichev-Saigo-Maeda fractional differentiation of the
generalized Struve function \eqref{1}. The following lemmas will be required
(see \cite{Kataria}).

\vskip 3mm

\begin{lemma}
\label{lem-5} Let $\lambda $, $\lambda ^{\prime },$ $\xi ,$ $\xi ^{\prime }$%
, $\gamma $, $\rho \in \mathbb{C}$ such that 
\begin{equation*}
\Re \left( \rho \right) >\max \left\{ 0,\Re \left( -\lambda +\xi \right)
,\Re \left( -\lambda -\lambda ^{\prime }-\xi ^{\prime }+\gamma \right)
\right\} .
\end{equation*}%
Then the following formula holds true: 
\begin{eqnarray}
&&\left( D_{0+}^{\lambda ,\lambda ^{\prime },\xi ,\xi ^{\prime },\gamma
}t^{\rho -1}\right) \left( x\right)  \label{D1} \\
&=&\frac{\Gamma \left( \rho \right) \Gamma \left( -\xi +\lambda +\rho
\right) \Gamma \left( \lambda +\lambda ^{\prime }+\xi ^{^{\prime }}-\gamma
+\rho \right) }{\Gamma \left( -\xi +\rho \right) \Gamma \left( \lambda
+\lambda ^{^{\prime }}-\gamma +\rho \right) \Gamma \left( \lambda +{\xi }%
^{^{\prime }}-\gamma +\rho \right) }x^{\lambda +\lambda ^{\prime }-\gamma
+\rho -1}.  \notag
\end{eqnarray}
\end{lemma}

\begin{lemma}
\label{lem-6} Let $\lambda $, $\lambda ^{\prime }$, $\xi $, $\xi ^{\prime }$%
, $\gamma $, $\rho \in \mathbb{C}$ such that 
\begin{equation*}
\Re \left( \rho \right) >\max \left\{ \Re \left( -\xi ^{\prime }\right) ,\Re
\left( \lambda ^{\prime }+\xi -\gamma \right) ,\Re \left( \lambda +\lambda
^{^{\prime }}-\gamma \right) +\left[ \Re \left( \gamma \right) \right]
+1\right\} .
\end{equation*}%
Then the following formula holds true: 
\begin{eqnarray}
&&\left( D_{-}^{\lambda ,\lambda ^{\prime },\xi ,\xi ^{\prime },\gamma
}t^{-\rho }\right) \left( x\right)  \label{D2} \\
&=&\frac{\Gamma \left( \xi ^{\prime }+\rho \right) \Gamma \left( -\lambda
-\lambda ^{\prime }+\gamma +\rho \right) \Gamma \left( -\lambda ^{\prime
}-\xi +\gamma +\rho \right) }{\Gamma \left( \rho \right) \Gamma \left(
-\lambda ^{\prime }+\xi ^{\prime }+\rho \right) \Gamma \left( -\lambda
-\lambda ^{\prime }-\xi +\gamma +\rho \right) }x^{\lambda +\lambda ^{\prime
}-\gamma -\rho }.  \notag
\end{eqnarray}
\end{lemma}

\begin{theorem}
\label{Th3}The following formula hold true:
\begin{eqnarray*}
&&\left( D_{0+}^{\lambda ,\lambda ^{\prime },\xi ,\xi ^{\prime },\gamma
}t^{\rho -1}~a\mathcal{W}_{p,b,c,\xi }^{\alpha ,\mu }\left( t\right) \right)
\left( x\right)  \\
&=&\frac{x^{\lambda +\lambda ^{^{\prime }}-\gamma +\rho +p}}{2^{p+1}} \\
&&\times _{4}\Psi _{5}\left[ 
\begin{array}{c}
\left( 1,1\right) ,\left( \rho
+p+1,2\right) ,\left( -\xi +\lambda +\rho +1,2\right) ,\\
\left( \mu ,\alpha
\right) ,\left( \frac{p}{\mu }+\frac{b+2}{2},a\right) ,\left( -\xi +\rho
+p+1,2\right) ,\left( \lambda +\lambda ^{\prime }-\gamma +\rho +p+1,2\right)
,
\end{array}%
\right.  \\
&&\left. 
\begin{array}{c}
\left( \lambda +\lambda ^{^{\prime }}+\xi ^{\prime }-\gamma +\rho
+p+1,2\right) ; \\ 
\left( \lambda +\xi ^{\prime }-\gamma +\rho +p+1,2\right) ;%
\end{array}%
\frac{-cx^{2}}{4}\right], 
\end{eqnarray*}
provided both the sides exists.
\end{theorem}

\begin{theorem}
\label{Th4} The following formula hold true: 
\begin{eqnarray*}
&&\left( D_{-}^{\lambda ,\lambda ^{\prime },\xi ,\xi ^{\prime },\gamma
}t^{\rho -1}~a\mathcal{W}_{p,b,c,\xi }^{\alpha ,\mu }\left( \frac{1}{t}%
\right) \right) \left( x\right)  \\
&=&\frac{x^{\lambda +\lambda ^{^{\prime }}-\gamma -\rho -p}}{2^{p+1}} \\
&&\times _{4}\Psi _{5}\left[
\begin{array}{c}
\left( 1,1\right) ,\left( \xi ^{\prime
}+\rho +p+1,2\right) ,\left( -\lambda -\lambda ^{\prime }+\gamma +\rho
+p+1,2\right) ,\\
\left( \mu ,\alpha \right) ,\left( \frac{p}{\mu }+\frac{b+2%
}{2},a\right) ,\left( \rho +p+1,2\right) ,\left( -\lambda +\xi ^{\prime
}+\rho +p+1,2\right) ,\end{array}{c}\right.  \\
&&\left. 
\begin{array}{c}
\left( -\lambda ^{\prime }-\xi +\gamma +\rho +p+1,2\right) ; \\ 
\left( -\lambda -\lambda ^{\prime }-\xi +\gamma +\rho +p+1,2\right) ;%
\end{array}%
\frac{-cx^{2}}{4}\right], 
\end{eqnarray*}
provided both the sides exists.
\end{theorem}

\end{document}